\documentclass[12pt, reqno,amscd, psamsfonts]{amsart}
\usepackage{amsthm, amsmath}
\usepackage{amssymb} 
\usepackage{amscd}
\usepackage{graphicx}
\usepackage{amsfonts}
\newtheorem{theorem}{Theorem}[section]

\newtheorem{lemma}[theorem]{Lemma}

\def\arrowdown#1#2{\Big\downarrow \rlap{$\vcenter{\hbox{$\scriptstyle#2$}}$}
{\hbox to -10pt{\hss{$\vcenter{\hbox{$\scriptstyle#1$}}$}}}}
\def\arrowup#1#2{\Big\uparrow \rlap{$\vcenter{\hbox{$\scriptstyle#2$}}$}
{\hbox to -10pt{\hss{$\vcenter{\hbox{$\scriptstyle#1$}}$}}}}
\begin{document}
\title [Infinite easier Waring constants for commutative rings]{Infinite easier Waring constants for commutative rings}

\author{Ted Chinburg}
\address{T.C.: Department of Mathematics\\University of
Pennsylvania\\Philadelphia, PA
19104-6395}
\email{ted@math.upenn.edu}

\thanks{The author was supported in part by NSF Grant DMS-0801030}

\date{\today}

\begin{abstract}
Suppose $n \ge 2$.  We show that there is no integer $v \ge 1$ such that for all commutative rings $R$ with identity,
every element of the subring $J(2^n,R)$ of $R$ generated by $2^n$-th powers can be written in the form 
$\pm f_1^{2^n} \pm \cdots  \pm f_v^{2^n}$ for some $f_1,\ldots,f_v \in R$ and some choice of signs.
\end{abstract}

\maketitle

\maketitle

%\begin{abstract}
%\end{abstract}

%\tableofcontents

\section{Introduction}
\label{s:intro}
\setcounter{equation}{0}

The object of this paper is to prove a result about easier Waring constants for commutative
rings which was announced over thirty years ago in \cite{C}.  
This result grew out of research of the author with Mel Henriksen in \cite{CH}.  It
is a pleasure to remember Mel's generosity and the excitement of working 
with him.  This paper is dedicated to Mel.  
%  Returning
%to this subject now is akin to finding a box of one's childhood toys in
%the attic.  Soon enough, toys became larger and more complicated.  But the
%joy of sharing ideas started with  Mel.  The author  
%remembers his generosity, and hopes that this paper will suggest
%better results to others.

Let $R$ be a commutative ring with identity, and suppose $k$ is a positive integer.
Define $J(k,R)$ to be the subring of $R$ generated by all $k^{th}$ powers.
 If there is an integer $v$ such that every element $f$ of $J(k, R)$ is of 
the form $$f = \sum_{i =1}^v \pm f_i^k$$
for some $f_1,\ldots,f_v \in R$ and some choice of signs, let $v(k, R)$ denote the smallest such $v$.  If no such $v$ exists, put 
$v(k, R) = \infty$. Let $V(k)$ be the sup, over all $R$, of $v(k, R)$.   %Let
%$\mathbb{Z}[\{x_i\}_{i=1}^\infty]$ be the ring of polynomials in countably
%many indeterminates with integer coefficients.  
Our main result is:

\begin{theorem}
\label{thm:main} For $n \ge 2$ one has 
$$V(2^n) = v(2^n,R_\infty) = \infty$$ when 
$R_\infty = \mathbb{Z}[\{x_i\}_{i=1}^\infty]$ is the ring of polynomials with integer coefficients in
countably many indeterminates.
%\begin{equation}
%\label{eq:arg}
%V(2^n) = v(2^n,\mathbb{Z}[\{x_i\}_{i=1}^\infty]) = \infty
%\end{equation}
\end{theorem}

To our knowledge,
this provides the first example of an integer $k$ for which  $V(k)$ is infinite.  

Concerning $k$
for which $V(k)$ is finite, Joly proved in \cite[Thm. 7.9]{J} that $V(2) = 3$.  
In \cite[Thm. 1]{C} it was shown that if $k$ is a prime which is not of the form $(p^{bc} - 1)/(p^c -1)$ for some prime $p$ and integers
$b \ge 2$ and $c \ge 1$, then $V(k)$ is finite.   This implies that $V(k)$ is finite for almost all primes $k$.
As of this writing, we do not know of further integers $k$ for which  $V(k)$ has been shown to
be finite.  The smallest integer $k > 2$ for which the results of \cite{C} show $V(k)$ to be finite  is $k = 11$. 

Striking quantitative results concerning upper bounds on $V(k)$ for various $k$, and on $v(k,R)$
for various $R$, have been proved by a number of authors including Car, Cherly, Gallardo, Heath-Brown, Newman, Slater, Vaserstein and others.
See  \cite{Car,Ch, GV, GH, NM,VW1, VW2, VC, VR} and their references.  

For $m \ge 1$ let $R_m = \mathbb{Z}[x_1,\ldots,x_m]$ be the ring of polynomials with integer coefficients in
$m$ commuting indeterminates.  By \cite[Theorems 3 and 4]{C},   $v(k,R_m)$ is finite for all $k$ and $m$.  By
 \cite[Prop. 7.12]{J},
 $$V(k) = \mathrm{sup}_{m \ge 1} v(k,R_m) = v(k,R_\infty).$$
 We show Theorem \ref{thm:main} by proving $\lim_{m \to \infty} v(k,R_m) = \infty$
when $k = 2^n > 2$.   We now summarize the strategy to be used in \S \ref{s:proofs} for 
bounding $v(k,R_m)$ from below in order to clarify how this approach might be applied for
other values of $k$. 

The strategy is to construct a surjection 
$$\pi:J(k,R_m) \to A$$
to an abelian group $A$ with the following property.  For $v \ge 1$, let $J(k,R_m)_v$ be the subset of
elements of $J(k,R_m)$ of the form $\sum_{i = 1}^v \pm f_i^k$ for some $f_i \in R_m$ and some 
choice of signs.  One would like to produce a $\pi$ such that $\pi(J(k,R_m)_v)$ has order
less than $A$ unless $v$ is at least some bound  which goes to infinity with $m$.

For $k = 2^n > 2$, the $\pi$ we construct in \S \ref{s:proofs} results from combining congruence
classes of the coefficients of high degree monomials which appear in the expansions of elements of $J(2^n,R_m)$.
The group $A$ is a vector space
over $\mathbb{Z}/2$ of dimension $m(m-1)/2$.  The $\pi$ we consider has the property that for $f_i \in R_m$, the value of
$\pi(f_i^{2^n})$ is $0$ if $f_i$ has odd constant term, and otherwise $\pi(f_i^{2^n})$ depends only on the coefficients mod $2$ of the homogeneous
degree $1$ part of $f_i$.  This means that the value of $\pi(\sum_{i = 1}^v \pm f_i^{2^n})$ depends
only on at most $vm$ elements of $\mathbb{Z}/2$, so that $$\# \pi(J(2^n,R_m)_v) \le 2^{vm}.$$  If $v = v(2^n,R_m)$,
so that $J(2^n,R_m)_v = J(2^n,R_m)$,
we must therefore have 
\begin{equation}
\label{eq:lowerbound}
vm = v(2^n,R_m) \cdot m \ge \mathrm{dim}_{\mathbb{Z}/2}(A) = m(m-1)/2
\end{equation}
since $\pi$ is surjective. 
 This produces the
lower bound
\begin{equation}
\label{eq:lower}
v(2^n,R_m) \ge (m-1)/2
\end{equation}
and implies Theorem \ref{thm:main}.

One can surely improve (\ref{eq:lower}), but we will not attempt to  optimize the above method in this paper.  
A systematic approach would be to 
consider which combinations of congruence classes of higher degree monomial coefficients of elements $f^{k}$ of $J(k,R_m)$
can be shown to depend only on the congruence classes of lower degree monomial coefficients of $f \in R_m$.   These
combinations should be chosen to be independent of one another, in the sense that they together
produce a surjection from $J(k,R_m)$ to a large abelian group $A$.

%From the vantage of thirty years, these problems are like childhood toys found in the
%attic.  Once all consuming, they now  mainly remind the author of the passage
%of time and of the excitement of talking about algebra with Mel Henriksen.  
% (compare \cite{toy}). 

\section{Proof of Theorem \ref{thm:main}}
\label{s:proofs}

Let $m \ge 1$ be fixed.  We will write polynomials
in $R_m = \mathbb{Z}[x_1,\ldots,x_m]$ in the form 
\begin{equation}
\label{eq:fform}
f = \sum_{\alpha} c_{f}(x^\alpha) x^\alpha
\end{equation}
 where 
 $$x^\alpha = \prod_{i = 1}^m x_i^{\alpha_i}$$
 is the monomial associated to a vector $\alpha = (\alpha_1,\ldots,\alpha_m)$
of non-negative integers 
and the integers $c_f(x^\alpha)$ are $0$ for almost all $\alpha$. 

\begin{lemma}
\label{lem:calc}  Suppose  $n \ge 2$ and 
$1 \le i < j \le m$. Then $c_{f^{2^n}}(x_i x_j)/2^n$
and $c_{f^{2^n}}(x_i^{2^{n-1}} x_j^{2^{n-1}})/2$ are integers.  One has 
\begin{equation}
\label{eq:congruence}
\frac{c_{f^{2^n}}(x_i x_j)}{2^n} + \frac{c_{f^{2^n}}(x_i^{2^{n-1}} x_j^{2^{n-1}})}{2} \equiv (c_f(1) + 1) c_f(x_i) c_f(x_j) \quad \mathrm{mod}\quad 2 \mathbb{Z}
\end{equation}
\end{lemma}

\begin{proof}  We first compute the coefficient $c_{f^{2^n}}(x_i x_j)$ of $x_i x_j$ in $f^{2^n}$.  Write
$$f = c_f(1) + t$$
where $t$ has constant term $0$. Then
\begin{equation}
\label{eq:f4}
f^{2^n} = c_f(1)^{2^n} + 2^n c_f(1)^{2^n-1} t + \frac{2^n (2^n -1)}{2} c_f(1)^{2^n - 2} t^2  + z
\end{equation}
where all the terms of $z \in R_m$ have degree larger than $2$.  
Here 
$$t \equiv \sum_{\ell = 1}^m c_f(x_\ell) x_\ell \quad \mathrm{mod}\quad \mathrm{terms\ of \ degree \ } \ge 2.$$
Because $i < j$, the coefficient of $x_i x_j$ in $t^2$ is $2c_f(x_i)c_f(x_j)$.  Putting
this into (\ref{eq:f4}), and noting that the coefficient of $x_i x_j$ in $t$ is $c_f(x_i x_j)$ by definition,
we conclude that 
\begin{equation}
\label{eq:xixj}
c_{f^{2^n}}(x_i x_j) = 2^n c_f(1)^{2^n -1} c_f(x_i x_j) + 2^n(2^n -1) c_f(1)^{2^n -2} c_f(x_i)c_f(x_j).
\end{equation}
Thus $2^n$ divides $c_{f^{2^n}}(x_i x_j)$. Because $n > 1$ and $a^s \equiv a$ mod $2$
for all $a \in \mathbb{Z}$ and $s \ge 1$, we find
\begin{equation}
\label{eq:nicer}
\frac{c_{f^{2^n}}(x_i x_j)}{2^n} \equiv c_f(1) (c_f(x_i x_j) + c_f(x_i) c_f(x_j)) \quad \mathrm{mod}\quad 2 \mathbb{Z}.
\end{equation}

We now consider the coefficient $c_{f^{2^n}}(x_i^{2^{n-1}} x_j^{2^{n-1}})$ of $x_i^{2^{n-1}} x_j^{2^{n-1}}$ in $f^{2^n}$.  
Write
\begin{equation}
\label{eq:ftwo}
f^{2^{n-1}} = \left (\sum_{\alpha} c_f(\alpha)^{2^{n-1}} (x^\alpha)^{2^{n-1}} \right ) + 2g   
\end{equation}
for some polynomial $g \in R_m$.  Then 
\begin{equation}
\label{eq:fagain}
f^{2^n} = (f^{2^{n-1}})^2 \equiv \left ( \sum_{\alpha} c_f(\alpha)^{2^{n-1}} (x^\alpha)^{2^{n-1}} \right )^2 \quad \mathrm{mod}\quad 4R_m.
\end{equation}
When one expands the square on the right side of (\ref{eq:fagain}), the coefficient of $x_i^{2^{n-1}} x_j^{2^{n-1}}$
is 
$$2c_f(1)^{2^{n-1}} c_f(x_i x_j)^{2^{n-1}} + 2 c_f(x_i)^{2^{n-1}} c_f(x_j)^{2^{n-1}}.$$
Because of the congruence (\ref{eq:fagain}) and the fact that $n > 1$, we conclude that $c_{f^{2^n}}(x_i^{2^{n-1}} x_j^{2^{n-1}})$ is divisible
by $2$, and 
\begin{eqnarray}
\label{eq:well}
\frac{c_{f^{2^n}}(x_i^{2^{n-1}} x_j^{2^{n-1}})}{2} &\equiv& c_f(1)^{2^{n-1}} c_f(x_i x_j)^{2^{n-1}} + c_f(x_i)^{2^{n-1}} c_f(x_j)^{2^{n-1}}  \quad \mathrm{mod} \quad 2\mathbb{Z}\nonumber \\
&\equiv& c_f(1) c_f(x_i x_j) + c_f(x_i) c_f(x_j)  \quad \mathrm{mod} \quad 2\mathbb{Z}
\end{eqnarray}
Adding (\ref{eq:nicer}) and (\ref{eq:well}) gives (\ref{eq:congruence}) and completes the proof.
 \end{proof}

\medbreak
\noindent {\bf Proof of Theorem \ref{thm:main}}
\medbreak
Fix $n > 1$ and $m \ge 2$ and suppose $1 \le i < j \le m$. By Lemma \ref{lem:calc}, there
is a unique homomorphism
\begin{equation}
\label{eq:pidef}
\pi_{i,j}:J(2^n,R_m) \to \mathbb{Z}/2
\end{equation}
which for $f \in R_m$ has the property that
$$\pi_{i,j}(f^{2^n}) = \left ( \frac{c_{f^{2^n}}(x_i x_j)}{2^n} + \frac{c_{f^{2^n}}(x_i^{2^{n-1}} x_j^{2^{n-1}})}{2} \right ) \quad \mathrm{mod}\quad 2$$
with the notation of Lemma \ref{lem:calc}.  The product of these homomorphisms over all pairs $(i,j)$
of integers such that $1 \le i < j \le m$ gives a homomorphism
\begin{equation}
\label{eq:pibigdef}
\pi:J(2^n,R_m) \to (\mathbb{Z}/2)^{\left ( m\atop 2 \right ) } = A
\end{equation}
Suppose we fix a pair $(i',j')$ of integers such that $1 \le i' < j' \le m$
and we let $f = x_{i'} + x_{j'}$. Formula (\ref{eq:congruence}) shows that
$\pi_{i,j}(f^{2^n}) = 0$ if $(i,j) \ne (i',j')$ while $\pi_{i',j'}(f^{2^n}) = 1$. It follows
that $\pi$ in (\ref{eq:pibigdef}) is surjective.  On the other hand, formula (\ref{eq:congruence})
shows that $\pi_{i,j}(f^{2^n}) = 0$ if $f$ has odd constant term $c_f(1)$, and that
otherwise $\pi_{i,j}(f^{2^n})$ depends only on the congruence classes mod $2$
of the linear terms in $f$.  Therefore the same is true of $\pi(f^{2^n})$.  As explained
in the second to last paragraph of the introduction, this
leads to the lower bounds (\ref{eq:lowerbound}) and (\ref{eq:lower}), which completes the proof.


\begin{thebibliography}{C-P-S1 }

\bibitem{Car}M.~Car, New bounds on some parameters in the Waring problem for polynomials over a finite field.  Finite fields and applications,  59--77, Contemp. Math., 461, Amer. Math. Soc., Providence, RI, 2008.

\bibitem{Ch} J.~Cherly, Sommes d'exponentielles cubiques dans l'anneau des $\mathrm{polyn\hat{o}mes}$ en une variable sur le corps \`a $2$ \'el\'ements, et application au probl\`eme de Waring. Ast\'erisque  No. 198-200  (1991), 83--96 (1992). 

\bibitem{C} T.~Chinburg, ``Easier'' Waring problems for commutative rings.  Acta Arith.  35  (1979), no. 4, 303--331.

\bibitem{CH} T.~Chinburg and M.~Henriksen,  Sums of $k$-th powers in the ring of polynomials with integer coefficients.  Acta Arith.  29  (1976), no. 3, 227--250.

\bibitem{GH} L.~Gallardo and D.~R.~Heath-Brown. Every sum of cubes in $\mathbb{F}_2[t]$ is a strict sum of 6 cubes.  Finite Fields Appl.  13  (2007),  no. 4, 981--987. 

\bibitem{GV}L.~Gallardo and L.~Vaserstein,  
The strict Waring problem for polynomial rings. 
J. Number Theory 128 (2008), no. 12, 2963Ð2972. 

\bibitem{J} J.~R.~Joly, Sommes de pusissances d-i\`emes dans un anneau commutatif, Acta Arith. 17 (1970), 37-114.

\bibitem{NM} D.~J.~Newman and M.~Slater, Waring's problem for the ring of polynomials.  J. Number Theory  11  (1979), no. 4, 477--487. 

\bibitem{VR} L.~Vaserstein, Ramsey's theorem and Waring's problem for algebras over fields.  The arithmetic of function fields (Columbus, OH, 1991),  435--441, Ohio State Univ. Math. Res. Inst. Publ., 2, de Gruyter, Berlin,  1992.

\bibitem{VC} L.~Vaserstein, Sums of cubes in polynomial rings.  Math. Comp.  56  (1991),  no. 193, 349--357.

\bibitem{VW1} L.~Vaserstein, Waring's problem for algebras over fields.  J. Number Theory  26  (1987),  no. 3, 286--298.

\bibitem{VW2} L.~Vaserstein, Waring's problem for commutative rings.  J. Number Theory  26  (1987),  no. 3, 299--307. 








\end{thebibliography}
\end{document}